\documentclass[11pt]{amsart}
\usepackage{graphicx}
\usepackage[mathscr]{eucal}
\usepackage{microtype}
\usepackage{amssymb}
\usepackage[parfill]{parskip}
\usepackage{mathrsfs}
\usepackage{xcolor}
\usepackage{color}
\usepackage[top=1in, bottom=1in, left=1in, right=1in]{geometry}
\usepackage{enumerate}
\usepackage{bm}
\usepackage[normalem]{ulem}
\usepackage[all]{xy}
\usepackage{tikz}

\usepackage{setspace} 

\allowdisplaybreaks

\newcommand{\plus}{{\cdot+\cdot}}

\newcommand{\zfc}{\mathnormal{\mathsf{ZFC}}}

\DeclareMathOperator{\supp}{supp}

\newtheorem{theorem}{Theorem}
\newtheorem{proposition}[theorem]{Proposition}
\newtheorem{lemma}[theorem]{Lemma}

\newtheorem{question}[theorem]{Question}

\theoremstyle{definition}

\author[D. Fern\'andez]{David J. Fern\'andez-Bret\'on}
\address{
Escuela Superior de F\'{\i}sica y Matem\'aticas\\
Instituto Polit\'ecnico Nacional\\
Av. Instituto Polit\'ecnico Nacional s/n Edificio 9, 
Col. San Pedro Zacatenco, Alcald\'{\i}a Gustavo A. Madero, 07738, CDMX, Mexico. 
}
\email{dfernandezb@ipn.mx}
\urladdr{https://dfernandezb.web.app}

\author[E. Sarmiento]{Eliseo Sarmiento Rosales}
\address{
Escuela Superior de F\'{\i}sica y Matem\'aticas\\
Instituto Polit\'ecnico Nacional\\
Av. Instituto Polit\'ecnico Nacional s/n Edificio 9, 
Col. San Pedro Zacatenco, Alcald\'{\i}a Gustavo A. Madero, 07738, CDMX, Mexico.
}
\email{esarmiento@ipn.mx}

\author[G. Vera]{Germ\'an Vera}
\address{
Escuela Superior de F\'{\i}sica y Matem\'aticas\\
Instituto Polit\'ecnico Nacional\\
Av. Instituto Polit\'ecnico Nacional s/n Edificio 9, 
Col. San Pedro Zacatenco, Alcald\'{\i}a Gustavo A. Madero, 07738, CDMX, Mexico. 
}
\email{gverame@ipn.mx}

\title[Owings-like theorems]{Owings-like theorems for infinitely many \\ colours or finite monochromatic sets}

\begin{document}

\begin{abstract}
Inspired by Owings's problem, we investigate whether, for a given an Abelian group $G$ and cardinal numbers $\kappa,\theta$,  every colouring $c:G\longrightarrow\theta$ yields a subset $X\subseteq G$ with $|X|=\kappa$ such that $X+X$ is monochromatic. (Owings's problem asks this for $G=\mathbb Z$, $\theta=2$ and $\kappa=\aleph_0$; this is known to be false for the same $G$ and $\kappa$ but $\theta=3$.) 
We completely settle the question for $\kappa$ and $\theta$ both finite (by obtaining sufficient and necessary conditions for a positive answer) and for $\kappa$ and $\theta$ both infinite (with a negative answer). Also, in the case where $\theta$ is infinite but $\kappa$ is finite, we obtain some sufficient conditions for a negative answer as well as an example with a positive answer.
\end{abstract}



\maketitle

\section{Introduction}

Let $G$ be a commutative semigroup (additively denoted), and let $\kappa,\theta$ be two cardinal numbers. We use the following (flavour of Hungarian) notation: the symbol
\begin{equation*}
G\rightarrow(\kappa)_\theta^\plus
\end{equation*}
denotes the (Ramsey-theoretic) statement that, for every colouring $c:G\longrightarrow\theta$, there exists a subset $X\subseteq G$ with $|X|=\kappa$ and $X+X$ monochromatic, where we define
\begin{equation*}
X+X=\{x+y\big|x,y\in X\}=\{x+y\big|x,y\in X\text{ are distinct}\}\cup\{2x\big|x\in X\};
\end{equation*}
the notation $G\nrightarrow(\kappa)_\theta^\plus$ will of course denote the negation of $G\rightarrow(\kappa)_\theta^\plus$. An old 1974 problem of J. Owings~\cite{owings-problem} asks whether $\mathbb N\longrightarrow(\aleph_0)_2^\plus$. It is relatively easy to obtain a $3$-colouring witnessing that $\mathbb N\nrightarrow(\aleph_0)_3^\plus$; N. Hindman~\cite{hindman-sumsets-3colours} has obtained a stronger result where the relevant $3$-colouring has the property that one of the colour classes has density 0.\footnote{In fact, N. Hindman's construction from~\cite{hindman-sumsets-3colours} characterizes, in a sense, how {\it small} can a colour class of an arbitrary $3$-colouring be while witnessing the same negative statement.} Surprisingly, Owings's original problem remains open. It is worth noting that the analogous questions where one requires the monochromatic set to only contain elements of the form $2x$, or only elements $x+y$ for $x\neq y$, are easy to answer in the positive (the first one is the pigeonhole principle and the second one follows directly from Ramsey's theorem). So the main combinatorial difficulty in all Owings-like questions is the combination arising from simultaneously considering both kinds of elements in our desired monochromatic set.

Another recent line of inquiry considers Owings type questions for other infinite groups\footnote{In this paper, all groups will be Abelian and will be written in additive notation, i.e., $+$ denotes the group operation, $-x$ denotes the inverse of the element $x$, and $0$ denotes the identity element.}, especially the additive group of the real numbers $\mathbb R$. Hindman, Leader and Strauss~\cite{hindman-leader-strauss} showed that it is consistent to have a finite $\theta$ such that $\mathbb R\nrightarrow(\aleph_0)_\theta^\plus$ (more concretely, if $|\mathbb R|=\aleph_n$ then one can take $\theta=9\cdot 2^{4+n}$); on the other hand, Leader and Russell~\cite{leader-russell} proved that, if $V$ is any $\mathbb Q$-vector space of dimension at least $\beth_\omega$, then $V\rightarrow(\aleph_0)_\theta^\plus$ for every finite $\theta$. A natural question that arises after consideration of these results is whether one can consistently show a positive result for the real line. 
Such a result was established by Komj\'ath, Leader, Russell, Shelah, D. Soukup, and Vidny\'anszky~\cite{6-authors-monochromatic-sumsets}, who established that, if the existence of an $\omega_1$-Erd\H{o}s cardinal is consistent, then so is the statement that $\mathbb R\longrightarrow(\aleph_0)_\theta^\plus$ holds for every finite $\theta$. Later on, J. Zhang~\cite{zhang-sumsets} removed the need for using large cardinals in the previous result, showing that it is consistent to have $|\mathbb R|=\aleph_{\omega+1}$ and $\mathbb R\longrightarrow(\aleph_0)_\theta^\plus$ for every finite $\theta$; furthermore, in that paper the statement $\mathbb R\longrightarrow(\aleph_0)_2^\plus$ is shown to be a $\zfc$ theorem.

As can be seen from the previous summary, questions about the statement $G\longrightarrow(\kappa)_\theta^\plus$ have been studied extensively for infinite $\kappa$ and finite $\theta$; in this setting, some surprising results have been obtained, while at the same time many questions remain open. In this paper we set to study the remaining combinations of $\kappa$ and $\theta$. Section 2 deals with the case where both $\kappa$ and $\theta$ are assumed to be finite; we obtain a characterization of those groups $G$ for which $G\longrightarrow(\kappa)_\theta^\plus$ holds regardless of the (finite) values of $\kappa,\theta$. Although a big part of the proof of this characterization relies on known theorems, or on standard adaptation of known techniques, there is at least one case where new ideas were needed. In Section 3 we study the case where both $\kappa$ and $\theta$ are infinite, and obtain a negative answer for all groups. While this result is not surprising and the proof is not groundbreaking, we include the main ideas for completeness. Finally, in Section 4---the main section---we tackle the case where $\theta$ is infinite but $\kappa$ is finite; in this case we obtain a mixed bag (where the result is positive for some groups but negative for others). We were unable to obtain a complete characterization like the one in Section 2, but we do obtain examples of groups where the answer is positive and a class of groups where it is negative. Our results suggest that, in order to fully characterize the groups satisfying this Ramsey-theoretic property, some relevant information might be gathered from the cardinalities of the groups $\{x\in G\big|4x=0\}$ and/or $G/\{x\in G\big|2x=0\}$.

\section*{Acknowledgements}

 This work is part of the third author's PhD dissertation, conducted under the supervision of the remaining two authors, and supported with PhD scholarship 744225 from Conahcyt. The first author was also partially supported by internal grant SIP-20230355 and Conahcyt grant CBF2023-2024-334. The second author was supported by internal grant SIP-20232086.

\section{Finitely many colours, finite monochromatic sets}

The main result of this section is a characterization of those groups satisfying $G\rightarrow(n)_r^\plus$ for all finite $n$ and $r$. We will be terse in our presentation in order to avoid excessive detail in routine calculations.

The first easy observation one can make is that Boolean groups $G$ must satisfy $G\nrightarrow(2)_2^\plus$, as can be seen simply by giving one colour to $0$ and the other to the remaining elements of $G$, since $2x=2y=0$ and $x+y\neq 0$ whenever $x,y$ are distinct elements of such a $G$. Our characterization is closely related to this observation.

\begin{theorem}\label{thm:sect2}
    Let $G$ be an Abelian group. Define $G_2=\{g\in G\big|2g=0\}$ and $2G=\{2g\big|g\in G\}$. Then, the following conditions are equivalent:
    \begin{enumerate}
        \item $G\rightarrow(n)_r^\plus$ for all finite $n,r$,
        \item $G/G_2$ is infinite,
        \item $2G$ is infinite.
    \end{enumerate}
\end{theorem}

In order to present the proof, we will utilize three preliminary lemmas, the proofs of which are mostly routine but we include at least the main ideas of each for the convenience of the reader.

\begin{lemma}\label{lem:order4}
Let $G$ be an abelian group such that $G_2=\{g\in G\big|2g=0\}$ is finite of cardinality $n$. Then, for any $c,d\in G$, the equation $2x=c$ has at most $n$ solutions, and the equation $4y=d$ has at most $n^2$ solutions (in particular, the number of elements of order 4 in $G$ is at most $n^2$, as they are contained within the solutions of the equation $4x=0$). 
\end{lemma}

\begin{proof}
Given a $c\in G$, either the equation $2x=c$ has no solutions, or admits at least one solution $x_0$. A routine calculation shows that any other solution must differ from $x_0$ by an element of order two, and so there are at most $n$ solutions. On the other hand, given a $d\in G$ notice that any solution $y_0$ to the equation $4y=d$ is also a solution of the equation $2x=c_0$, where $c_0=2y_0$ is a solution of $2x=d$. Hence there are at most $n$ possible values for $c_0$, and for each of these there are at most $n$ possible values for $y_0$, yielding at most $n^2$ possible solutions overall for the equation $4y=d$.
\end{proof}

\begin{lemma}\label{lem:uniqueness}
Let $G$ be an infinite group with only finitely many elements of order $2$. Then, there exists a sequence of elements of $G$, $(g_n\big|n\in\mathbb N)$, with the property that, whenever there is an equality
\begin{equation*}
\sum_{i=1}^n\varepsilon_i g_{k_i}=\sum_{j=1}^m\delta_j g_{l_j}
\end{equation*}
where $k_1<\cdots<k_n$, $l_1<\cdots<l_m$, and $\varepsilon_i,\delta_j\in\{1,2,4\}$, then it must be the case that $n=m$, $k_i=l_i$ for all $i\in\{1,\ldots,n\}$, and $\varepsilon_i=\delta_i$ for all $i<n$.
\end{lemma}

\begin{proof}
If there is an $a\in G$ of infinite order, then the sequence $g_n=na$ satisfies the given condition. So we assume that $G$ is a torsion group with finitely many elements of order $2$. Then there are also finitely many elements of order $4$ by Lemma~\ref{lem:order4}. Let $g_0\in G$ be any element such that $o(g_0)\notin\{1,2,4\}$. Recursively pick $g_n\in G\setminus H_k$, where $H_k=\langle g_i\mid i<k\rangle$ is the subgroup generated by the $g_i$ ($i<k$), such that $g_n$ is not a solution to any equations $2x=c$, $4x=c$ for $c\in H_k$ (this can be done by Lemma~\ref{lem:order4} and the fact that $H_k$ is finite). In other words, $\{g_n,2g_n,4g_n\}\cap H_k=\varnothing$. This defines the sequence $(g_n \mid n \in \mathbb{N})$. The desired conclusion stems from the fact that, given any equation
\begin{equation*}
\sum_{i=1}^n \varepsilon_i g_{k_i} = \sum_{j=1}^m \delta_j g_{l_j}
\end{equation*}
as in the statement of the lemma, one can always cancel equal terms on both sides and, should any terms still survive, move the term with highest index to one side of the equation and the remaining terms to the other side. This way we get an equation where one of $g_{n}, 2g_{n}, 4g_{n}$ equals some combination of the previous $g_k$, contradicting the choice of $g_n$.
\end{proof}

In order to handle expressions like the one considered in the previous paragraph, we introduce the following notation: given a finite sequence of integers $\vec{\varepsilon}=(\varepsilon_1,\ldots,\varepsilon_m)$, and a finite sequence of group elements $\vec{g}=(g_1,\ldots,g_n)$, the notation $\vec{\varepsilon}*\vec{g}$ will denote the element
\begin{equation}
\vec{\varepsilon}*\vec{g}=\sum_{i=1}^{\min\{n,m\}}\varepsilon_i g_i.
\end{equation}

\begin{lemma}\label{lem:secuencia-orden-2-4}
    Let $G$ be a group such that $2G=\{2g\big|g\in G\}$ has infinitely many elements of order $2$. Define $\varepsilon_0=3,\varepsilon_1=1,\delta_0=0,\delta_1=2$ and, for a function $h:\{1,\ldots,n\}\longrightarrow\{0,1\}$, denote $\vec{\varepsilon_h}=(\varepsilon_{h(1)},\ldots,\varepsilon_{h(n)})$, and similarly for $\vec{\delta_h}$. Then, there exists a sequence of elements $\vec{z}=(z_n\big|n\in\mathbb N)$ such that, for any choice of distinct indices $i_1,\ldots,i_n,j_1,\ldots,j_n\in\mathbb N$, and for any two functions $f,g:\{1,\ldots,n\}\longrightarrow\{0,1\}$, we have that $\vec{\varepsilon_f}*(z_{i_1},\ldots,z_{i_n})+\vec{\delta_f}*(z_{j_1},\ldots,z_{j_n})=\vec{\varepsilon_g}*(z_{i_1},\ldots,z_{i_n})+\vec{\delta_g}*(z_{j_1},\ldots,z_{j_n})$ if and only if $f=g$.
\end{lemma}

\begin{proof}
     First pick elements $g_0,g_1,\ldots\in 2G$ by recursion in such a way that each $g_n$ has order 2 and does not belong to the subgroup generated by the previous $g_k$, $k<n$. Now for each $n$ let $z_n$ be such that $2z_n=g_n$.
     
    We now prove that the $z_n$ are as required. In the nontrivial direction, proceed by induction on $n$. For $n=1$, the only way (up to symmetry) we can have two distinct functions $f,g:\{1\}\longrightarrow\{0,1\}$ is by having $f(1)=0$ and $g(1)=1$. Then the equation $\vec{\varepsilon_f}*(z_{i_1})+\vec{\delta_f}*(z_{j_1})=\vec{\varepsilon_g}*(z_{i_1})+\vec{\delta_g}*(z_{j_1})$ amounts to $3z_{i_1}+0z_{j_1}=1z_{i_1}+2z_{j_1}$, which readily implies $g_{i_1}=2z_{i_1}=2z_{j_1}=g_{j_1}$, contradicting our choice of $g_k$. Now suppose that the result is true for each $k<n$, with $n\geq 2$. Suppose we have pairwise distinct indices $i_1,\ldots,i_n,j_1,\ldots,j_n$ and two functions $f,g:\{1,\ldots,n\}\longrightarrow\{0,1\}$ such that $\vec{\varepsilon_f}*\vec{x}+\vec{\delta_f}*\vec{y}=\vec{\varepsilon_g}*\vec{x}+\vec{\delta_g}*\vec{y}$. If there is at least one $k\in\{1,\ldots,n\}$ such that $f(k)=g(k)$, then we can cancel out the terms $\varepsilon_{f(k)}z_{i_k}+\delta_{f(k)}z_{j_k}$ from both sides of the equation, and use the inductive hypothesis to conclude that $f=g$. So we may assume that $f(k)\neq g(k)$ for all $k\in\{1,\ldots,n\}$. This means that, for each $k$, one of the sides of the equation contains the term $3z_{i_k}$, and the other side contains the terms $z_{i_k}+2z_{j_k}$. Hence, by simply subtracting $z_{i_1}+\cdots+z_{i_n}$ to both sides of the equation, we obtain an equation of the form
    \begin{equation*}
    \sum_{k\in F}2z_{i_k}+\sum_{i\in G}2z_{j_k}
    =\sum_{k\in F}2z_{j_k}+\sum_{i\in G}2z_{i_k},
    \end{equation*}
    where $F,G$ are two disjoint sets such that $F\cup G=\{1,\ldots,n\}$. Since each $g_k=2z_k$, we have obtained an equation where the $g_k$ with largest index can be moved to one side of the equation, and all remaining terms to the other side. But then some $g_n$ equals a combination of the previous $g_k$, a contradiction.
\end{proof}

We are now ready for the proof of the main result of this section.

\begin{proof}[Proof of Theorem~\ref{thm:sect2}]
    Note that points (2) and (3) are easily seen to be equivalent by considering the group morphism $g\longmapsto 2g$. This morphism has image $2G$ and kernel $G_2$, so by the first isomorphism theorem there is an isomorphism between $G/G_2$ and $2G$; in particular, these two sets are equipotent.
    
    Let us now prove that (1) implies (3). Suppose $2G$ is finite, say $2G=\{h_0,\ldots,h_{k-1}\}$. Define $c:G\longrightarrow k+1$ by
    \begin{equation*}
        c(g)=\begin{cases}
            i,\ \text{ if }g=h_i; \\
            k,\ \text{ if }g\notin H
        \end{cases}
    \end{equation*}
    If there were two distinct $x,y\in G$ such that $\{x,y\}+\{x,y\}=\{2x,2y,x+y\}$ was monochromatic, we would need to have $2x=2y$ since $2x,2y\in 2G$ and $c$ is injective in $2G$. This implies that $x+y\in 2G$ also, and $2x=2y=x+y$, hence $x=y$, a contradiction.

    The rest of the proof, devoted to show that (3) implies (1), splits into three cases, which we present in increasing order of difficulty (the first two cases are fairly routine).

    The first case is when the infinite group $2G$ contains an element of infinite order $g$. In this case, $g$ generates a subsemigroup isomorphic to $\mathbb N$. Since the statement $\mathbb N\longrightarrow(n)_r^\plus$ holds for all finite $n,r$ (for a proof see e.g. Hindman~\cite[Theorem 2.1]{hindman-sumsets-3colours}, although Hindman himself attributes this result to R. Rado and W. Deuber), it follows that $G\rightarrow(n)_r^\plus$.

    The second case is when the (infinite) group $2G$ is a torsion group, and contains only finitely many elements of order $2$. In this case, we use Lemma~\ref{lem:uniqueness} to obtain a sequence of elements $g_n$ such that each $g\in G$ of the form $\sum_{i=1}^n\varepsilon_i g_{k_i}$, with $k_1 <\cdots<k_n$ and $\varepsilon_i\in\{1,2,4\}$, can be represented uniquely as such an expression on the $g_n$. Having this sequence in hand, one just needs to follow the main idea from the main theorem (Theorem 2.2) in~\cite{leader-russell}: consider the patterns given by $\vec{\varepsilon_i}=(\underbrace{4,\ldots,4}_{r-i\text{ times}},\underbrace{2,\ldots,2}_{2i\text{ times}})$ for $i\in(r+1)$,
    define a new colouring $d$ of $2r$-tuples of natural numbers by
    \begin{equation*}
    d(n_1,\ldots,n_{2r})=(c(\vec{\varepsilon_0}*(g_{n_1},\ldots,g_{n_{2r}})),c(\vec{\varepsilon_1}*(g_{n_1},\ldots,g_{n_{2r}})),\ldots,c(\vec{\varepsilon_r}*(g_{n_1},\ldots,g_{n_{2r}}))),
    \end{equation*}
    use Ramsey's theorem~\cite{ramsey} to get an infinite $Y\subseteq\mathbb N$ such that $[Y]^{2r}$ is $d$-monochromatic, say on colour $(t_0,\ldots,t_r)$, and pigeonhole the $t$ to get two patterns, $t_i$ and $t_j$, of the same colour. Then, defining
    \begin{equation*}
    x_k=2g_{a_1}+\cdots+2g_{a_{r-j}}+2g_{b_{(k-1)(j-i)+1}}+\cdots+2g_{b_{k(j-i)}}+g_{c_1}+\cdots+g_{c_{2i}},
    \end{equation*}
    where the $b_k$ vary and lie between the $a_k$ and the $c_k$, one obtains the set $X=\{x_k\big|k\in\{1,\ldots,n\}\}$ satisfying that $X+X$ is $c$-monochromatic. The details are best left to the reader.

    Finally, in the case where $2G$ contains infinitely many elements of order $2$, we take a sequence $(z_n\big|n\in\mathbb N)$ and define $\varepsilon_0,\delta_0,\varepsilon_1,\delta_1$ as in Lemma~\ref{lem:secuencia-orden-2-4}. Given a finite colouring $c$ we define a colouring $d$ on $n$-tuples of natural numbers by letting $d(k_1,\ldots,k_n)=c(2z_{k_1}+\cdots+2z_{k_n})$. Ramsey's theorem provides us with an infinite $Y\subseteq\mathbb N$ such that $[Y]^n$ is $d$-monochromatic, say on colour $l$. Pick $2n$ distinct indices $i_1,\ldots,i_n,j_1,\ldots,j_n\in Y$ and, for each $f:\{1,\ldots,n\}\longrightarrow\{0,1\}$, define $x_f=\vec{\varepsilon_f}*(z_{i_1},\ldots,z_{j_n})+\delta_f*(z_{j_1},\ldots,z_{j_n})$. The choice of $z_n$ ensures that the $x_f$ are mutually distinct; moreover, for two distinct $f,g:\{1,\ldots,n\}\longrightarrow\{0,1\}$, we have
    \begin{eqnarray*}
    2x_f=2z_{i_1}+\cdots+2z_{i_n}, \\
    x_f+x_g=2z_{t_1}+\cdots+2z_{t_n},
    \end{eqnarray*}
    where $t_k=i_k$ if $f(k)=g(k)$, and $t_k=j_k$ if $f(k)\neq g(k)$. This implies $c(2x_f)=c(x_f+x_g)=l$, so $X=\{x_f\big|f:\{1,\ldots,n\}\longrightarrow\{0,1\}\}$ contains $2^n$ elements and satisfies that $X+X$ is $c$-monochromatic.    
\end{proof}

The idea used in the last case of the previous proof will be used again in Proposition~\ref{prop:order-4}.

\section{Infinitely many colours, no infinite monochromatic set}

The main result of this section is that $G \nrightarrow (\aleph_0)_{\aleph_0}^\plus$ for all abelian groups $G$. Although most of the proof involves routine ideas, we include it here for the convenience of the reader.

 We begin by recalling the useful fact that every Abelian group can be embedded in a direct sum $\bigoplus_{\alpha<\kappa}G_\alpha$, where each $G_\alpha$ is a copy of either $\mathbb Q$, or a Pr\"ufer group $\mathbb Z[p^\infty]$ for some prime number $p$ (for details see e.g.~\cite[Th. 9.23 and 9.14]{rotman}); furthermore if $G$ is uncountable then the index set $\kappa$ of the direct sum can be assumed to be $|G|$. Hence each element of a group can be thought of as a member of some such direct sum; in order to uniformize notation and not worry about the precise nature of the summand $G_\alpha$, we will use the symbol\footnote{Note that $\mathbb G$ is a countable subgroup of the 1-dimensional torus $\mathbb T=\mathbb R/\mathbb Z$, in which each $\mathbb Z[p^\infty]$, as well as $\mathbb Q$, can be embedded (the Pr\"ufer groups $\mathbb Z[p^\infty]$ may in fact be {\it defined} as the subgroup of $\mathbb T/\mathbb Z$ generated by the equivalence classes of $\frac{1}{p^n}$, where $n$ ranges over all natural numbers, whereas for $\mathbb Q$ it suffices to consider $q\longmapsto q\sqrt{2}$).} $\mathbb G=\mathbb Q[\sqrt{2}]/\mathbb Z$ and hence think of each uncountable (Abelian) group $G$ as a subgroup of $\bigoplus_{\alpha<\kappa}\mathbb G$, where $\kappa=|G|$. Given an element $g=(g_\alpha\big|\alpha<\kappa)\in G$, we define the {\it support} of $g$ by $\supp(g)=\{\alpha<\kappa\big|g_\alpha\neq 0\}$. With these tools, we now prove the main result of the section.
 
\begin{theorem}\label{thm:no-infinite}
For every group $G$, there exists a colouring $c:G\longrightarrow C$, with $C$ countably infinite, such that for every infinite subset $X\subseteq G$ the set $X+X$ is not monochromatic.
\end{theorem}

\begin{proof}
Since the result is obvious for a (finite or infinite) countable $G$ (simply give each element of $G$ a different colour), it suffices to consider the case $G=\bigoplus_{\alpha<\kappa}\mathbb G$, where $\kappa$ is an uncountable cardinal (note that the relevant property is inherited to any subgroup $H\leq G$). Let $C$ be the set of finite sequences of nonzero elements of $\mathbb G$, and (noting that $C$ is countable) define the colouring $c:G\longrightarrow C$ by the formula
\begin{equation*}
c(g_\alpha\big|\alpha<\kappa)=(g_{\alpha_1},\ldots,g_{\alpha_m}),
\end{equation*}
where $\alpha_1,\ldots,\alpha_m$ are the elements of $\supp(g)$ enumerated increasingly. In other words, $c(g)$ is the (finite) sequence of nonzero entries of $g$.

Suppose that $X\subseteq G$ is an infinite set such that $X+X$ is $c$-monochromatic, say on colour $(x_1,\ldots,x_n)$. Given two distinct elements $g,h\in X$, we note that $\supp(g)\bigtriangleup\supp(h)\subseteq\supp(g+h)\subseteq\supp(g)\cup\supp(h)$, and so $|\supp(g)\bigtriangleup\supp(h)|\leq n$. Moreover, $|\supp(g)|\leq |\supp(h)|+|\supp(g)\setminus \supp(h)|\leq |\supp(h)|+|\supp(g)\bigtriangleup \supp(h)| \leq |\supp(h)|+n$. 
This implies that the set of natural numbers $\{|\supp(g)|\big|g\in X\}$ is bounded by $n+|\supp(h)|$ where $h\in X$ is any fixed element, therefore, by the pigeonhole principle, we may assume that all elements of $X$ have supports of some fixed cardinality $s$; suppose furthermore that $s$ is as small as possible. Note that, if $(y_1,\ldots,y_s)$ is the sequence of nonzero elements of some $g\in X$, then for each $i\leq s$ it must be the case that $2y_i$ equals either $0$, or some $x_j$ for $j\leq n$. By Lemma~\ref{lem:order4}, each equation $2y=x$ for a fixed $x$ has at most two solutions in $\mathbb T$ ($\mathbb T$ has exactly one element of order $2$), and hence there are only finitely many (at most $2^s$) possible choices for the sequence $(y_1,\ldots,y_s)$. So again by the pigeonhole principle we further assume that all elements of $X$ have the same sequence $(y_1,\ldots,y_s)$ of nonzero entries. In particular, no two distinct elements of $X$ can have the same support set, and so $\{\supp(g)\big|g\in X\}$ is an infinite family of sets of size $s$. Hence, by the $\Delta$-system lemma\footnote{Also known as  {\it the sunflower lemma}. For a proof, see~\cite[p. 421]{komjath-totik}.}~{\cite[p. 107, statement 1]{komjath-totik}}, we may assume also that the supports of elements of $X$ form a $\Delta$-system, say with root $R\subseteq\kappa$. Further applications of the pigeonhole principle allow us to assume that the sequence of $R$-entries $(g_{\alpha_1},\ldots,g_{\alpha_{r}})$ (where $r=|R|$ and $\alpha_1<\ldots<\alpha_r$ are the elements of $R$) is constant across $X$.

We claim that for each $\alpha\in R$, it must be the case that $\alpha\in\supp(g+h)$ for any two $g,h\in X$. If not, then we must have $g_\alpha+h_\alpha=0$, but since we are assuming $g_\alpha=h_\alpha$ this implies either $g_\alpha=h_\alpha=0$, a contradiction, or $g_\alpha=h_\alpha=\frac{1}{2}$. However, this allows us to ``ignore'' the $\alpha$-th entry:  for each $g\in X$, let $\tilde{g}$ denote the element with the same entries as $g$, except in the $\alpha$-th entry, where $\tilde{g}_\alpha=0$. Then the set $\tilde{X}=\{\tilde{g}\big|g\in X\}$ is another infinite set, with $\tilde{X}+\tilde{X}$ monochromatic (in the same colour as $X+X$) and all elements $\tilde{g}$ satisfying $|\supp(\tilde{g})|=s-1$, contradicting the minimality of $s$. Hence, this second case also leads to a contradiction, and we conclude that $2g_\alpha=2h_\alpha=g_\alpha+h_\alpha\neq 0$. Therefore, $R=\supp(g)\cap\supp(h)\subseteq\supp(g+h)$, which implies that $\supp(g+h)=\supp(g)\cup\supp(h)$. This itself is a contradiction, as it implies $n=|c(g+h)|=|\supp(g+h)|=r+2(s-r)$,  while simultaneously $n=|c(2g)|=|\supp(2g)|\leq|\supp(g)|=s$, since $\supp(2g)\subseteq\supp(g)$. The last two (in)equalities show that $2s-r\leq s$, implying $s\leq r$ and so any two elements of $X$ must have the same support, which we had already established as impossible.
\end{proof}

\section{Infinitely many colours, finite monochromatic sets}

This is the section containing the two main results of the paper, as these are the ones whose proof does not consist of routine computations. Unlike in Section 2, we were unable to exactly characterize the groups satisfying the statement $G\rightarrow(\kappa)_{\theta}^\plus$ (for infinite $\kappa$ and $\theta$); however, we have a sufficient condition for this to fail as well as an example where it holds.




\begin{theorem}\label{thm:no-order-4}
Let $G$ be a group satisfying one of the following properties:
\begin{enumerate}
    \item either $G$ is countable, or
    \item $G$ is a torsion group without elements of order $4$, or
    \item $G$ has no elements of order $2$.
\end{enumerate}
Then $G\nrightarrow(2)_{\aleph_0}^\plus$
\end{theorem}

\begin{proof}
The case where $G$ is countable is straightforward (simply give a different colour to each element of $G$, as we did in the proof of Theorem~\ref{thm:no-infinite}), so we begin by assuming that $G$ is uncountable, and thus a subgroup of $\bigoplus_{\alpha<\kappa}\mathbb G$ where $\kappa=|G|$. If $C$ is the (countable) set of finite sequences of nonzero elements of $\mathbb G$, we define $c:G\longrightarrow C$ by letting $c(g)$ be the (finite) sequence of nonzero entries of $g$.

Suppose there were two distinct elements $g=(g_\alpha\big|\alpha<\kappa),h=(h_\alpha\big|\alpha<\kappa)$ such that $\{g,h\}+\{g,h\}=\{2g,2h,g+h\}$ were $c$-monochromatic. We begin by showing that we can assume that none of the entries $g_\alpha,h_\alpha$ has order $4$. In case $G$ is a torsion group without elements of order $4$, simply note that, in case some $g_\alpha$ had order $4$, there would be some integer multiple of $g$ of order $4$, contradicting the hypothesis about $G$. If, on the other hand, $G$ (is not necessarily a torsion group, but it) lacks elements of order $2$, we argue as follows: suppose that some entries of $g$, or of $h$, have order a power of $2$, and let $k$ be sufficiently large that both $2^k g$ and $2^k h$ are still nonzero, but do not have any entry left with order a power of $2$ (such a $k$ exists because, by hypothesis, the orders of $g$ and $h$ cannot be powers of $2$, and they only have finitely many nonzero entries). If $2^k g=2^k h$, then either $g-h$ is zero, or an element whose order divides $2^k$; the assumption that $G$ lacks elements of order two makes the latter case impossible, and hence we must have $g=h$, a contradiction. Therefore $2^k g\neq 2^k h$ and the set $\{2(2^k g),2(2^k h), 2^k g+2^k h\}$ still is $c$-monochromatic.

So we may assume that there exist two distinct elements $g=(g_\alpha\big|\alpha<\kappa),h=(h_\alpha\big|\alpha<\kappa)$ such that none of the entries $g_\alpha,h_\alpha$ has order $4$, and furthermore the set $\{2g,2h,g+h\}$ is $c$-monochromatic, say on colour $(x_1,\ldots,x_n)$. We show, by induction on $\alpha$, that $g_\alpha=h_\alpha$ for all $\alpha<\kappa$, hence $g=h$, leading us to a contradiction.

Suppose, then, that $g_\xi=h_\xi$ for all $\xi<\alpha$. If $g_\alpha=h_\alpha=0$ we are done, so assume without loss of generality that $g_\alpha\neq 0$. If $h_\alpha=0$, then the value $g_\alpha$ is one of the non-zero entries of $g+h$, and so there is an $i\leq n$ such that $x_i=g_\alpha$. Since $c(g+h)=c(2g)$, the element $2g$ has the exact same sequence of non-zero entries as $g+h$, so there must be some $g_\beta$, with $\beta\geq\alpha$, such that $2g_\beta=x_i=g_\alpha$. This implies in particular that $o(g_\alpha)\neq 2$, since otherwise $g_\beta$ would have order $4$, a contradiction. Hence $2g_\alpha\neq 0$ and so $2g_\alpha$ is the $i$-th nonzero entry of both $2g$ and $g+h$, implying that $2g_\alpha=g_\alpha+h_\alpha$, and therefore $h_\alpha=g_\alpha\neq 0$.

From the previous paragraph, we can conclude that both $g_\alpha$ and $h_\alpha$ are nonzero. One possibility is that $g_\alpha=h_\alpha=\frac{1}{2}$, in which case we are done; so we may assume without loss of generality that $o(g_\alpha)\neq 2$. Then, again, this means that $2g_\alpha\neq 0$ and therefore $x_i=2g_\alpha=g_\alpha+h_\alpha$, implying that $g_\alpha=h_\alpha$. In all of the possible cases we have been able to conclude that $g_\alpha=h_\alpha$, and we are done.
\end{proof}

It is unclear to the authors if simply requiring that $G$ lacks elements of order $4$ (cf. question~\ref{q:order4}) is enough to ensure that $G\nrightarrow(2)_{\aleph_0}^\plus$. For example, within the group $(\mathbb Z/4\mathbb Z)\times(\mathbb Z/4\mathbb Z)\times\mathbb Z$, consider the subgroup $G$ generated by $\{(1,0,1),(2,0,0),(0,2,0)\}$, which contains elements of order $2$ but lacks elements of order $4$. The elements $g=(1,0,1)$ and $h=(3,2,1)$ both belong to $G$ and are such that the set $\{2g,2h,g+h\}$ is monochromatic for the colouring described in the proof of Theorem~\ref{thm:no-order-4}. However, $G$ is isomorphic to $\tilde{G}=\mathbb Z\times(\mathbb Z/2\mathbb Z)\times(\mathbb Z/2\mathbb Z)$, and $\tilde{G}$ does not contain two distinct elements $g,h$ such that $\{2g,2h,g+h\}$ is monochromatic for said colouring. So it is possible that, in order to establish a more general result, one might need to carefully choose the embedding of $G$ into $\bigoplus\mathbb G$.

To finish the paper, we exhibit an uncountable group with lots of elements of order $4$ for which the positive result follows. Recall that the cardinal numbers $\beth_\alpha(\theta)$ are defined recursively by $\beth_0(\theta)=\theta$, $\beth_{\alpha+1}(\theta)=2^{\beth_\alpha(\theta)}$, and $\beth_\alpha(\theta)=\sup\{\beth_\xi(\theta)\big|\xi<\alpha\}$ if $\alpha$ is a limit ordinal.

\begin{proposition}\label{prop:order-4}
Let $\theta$ be an arbitrary infinite cardinal. Let $\kappa=\beth_\omega(\theta)$, and consider the group $G=\bigoplus_{\alpha<\kappa}(\mathbb Z/4\mathbb Z)$. Then, for every finite $n$, we have $G\rightarrow(n)_{\theta}^\plus$.
\end{proposition}

\begin{proof}
In order to fix notation, denote by $e_\alpha$ the element of $G$ whose $\alpha$-th entry equals $1$, with all other entries equal to $0$. Hence each element of $G$ is a finite $\mathbb Z$-linear combination of the $e_\alpha$. Fixing an $n\in\mathbb N$, we define $\varepsilon_0=1$, $\delta_0=2$, $\varepsilon_1=3$, $\delta_1=0$, and we proceed to define $2^n$ finite sequences of integers of length $2n$, as follows: for each $f:\{1,\ldots,n\}\longrightarrow\{0,1\}$ let
\begin{equation*}
\vec{s_f}=(\varepsilon_{f(1)},\delta_{f(1)},\ldots,\varepsilon_{f(n)},\delta_{f(n)}).
\end{equation*}
Given an arbitrary colouring $c:G\longrightarrow\theta$, we define another colouring $d:[\kappa]^n\longrightarrow\theta$ by
\begin{equation*}
d(\alpha_1,\ldots,\alpha_n)=c(2e_{\alpha_1}+\cdots+2e_{\alpha_n})
\end{equation*}
whenever $\alpha_1<\cdots<\alpha_n$. By the Erd\H{o}s--Rado theorem~\cite[Theorem 9.6]{jech}, there exists an infinite set $Y$ such that $[Y]^n$ is $d$-monochromatic, say on colour $k$. Pick $2n$ distinct elements $\alpha_1,\ldots,\alpha_{2n}\in Y$, with $\alpha_1<\cdots<\alpha_{2n}$, and for each $f:\{1,\ldots,n\}\longrightarrow\{0,1\}$ let $x_f=\vec{s_f}*\vec{\alpha}$, where $\vec{\alpha}=(\alpha_1,\ldots,\alpha_{2n})$. The reader will gladly verify that, if $f,g:\{1,\ldots,n\}\longrightarrow\{0,1\}$ are distinct, then
\begin{eqnarray*}
2x_f=2e_{\alpha_1}+2e_{\alpha_3}+\cdots+2e_{\alpha_{2n-1}}, \\
x_f+x_g=2e_{\alpha_{1+i_0}}+2e_{\alpha_{3+i_1}}+\cdots+2e_{\alpha_{2n-1+i_n}},
\end{eqnarray*}
where $i_j=|f(j)-g(j)|$, so that 
$$c(2x_f)=d(\alpha_1,\alpha_3,\ldots,\alpha_{2n-1})=k$$
and
$$c(x_f+x_g)=d(\alpha_{1+i_0},\alpha_{3+i_1},\ldots,\alpha_{2n-1+i_n})=k.$$
Hence if we let $X=\{x_f\big|f:\{1,\ldots,n\}\longrightarrow\{0,1\}\}$, then $|X|=2^n$ and $X+X$ is monochromatic in colour $k$. Since the values of $2^n$, as $n$ varies, can be arbitrarily large, the conclusion is that one can get such sets $X$ of any possible finite size, and we are done.
\end{proof}

Note, in particular, that for every colouring of the group from Proposition~\ref{prop:order-4} there are arbitrarily large (but finite) subsets $X$ such that $X+X$ is monochromatic. For a fixed size of the required monochromatic set, however, it is very likely that the group from the previous proposition is overkill in terms of size. For example, upon fixing $n$ and $\theta$, one only needs to take $\kappa=\beth_n(\theta)^+$ for the argument in the proof of Proposition~\ref{prop:order-4} to work and be able to obtain a monochromatic set $X+X$ with $|X|=2^n$ (the key point here being the application of the Erd\H{o}s--Rado theorem). Obtaining more precise information along this lines seems to be an interesting question that we leave open.

\begin{question}\label{q:order4}
    \footnote{{\it Added in print:} This question has now been answered, in the aﬃrmative, by I. Leader and K. Williams (arXiv:2407.03938).}Let $G$ be an uncountable group without elements of order $4$. Is it the case that $G\nrightarrow(2)_{\aleph_0}^\plus$?
\end{question}

\begin{question}
Is it possible to characterize (in terms of $|G|$, $|G_2|$, $|G_4|$ and/or possibly $|G/G_2|$, where $G_d=\{x\in G\big|dx=0\}$) precisely those Abelian groups $G$ satisfying (whether for all $n$, or for some specific one) $G\rightarrow(n)_\theta^\plus$?
\end{question}

\end{document}